\documentclass[a4paper, 11pt]{amsart}

\usepackage{cite}
\usepackage{amsmath}
\usepackage{amssymb}
\usepackage{amsthm}   
\usepackage{empheq}   
\usepackage{graphicx}
\usepackage{color}
\usepackage{appendix}

\newtheorem{theorem}{Theorem}[section]
\newtheorem{lemma}[theorem]{Lemma}

\theoremstyle{definition}

\theoremstyle{remark}

\numberwithin{equation}{section}

\begin{document}

\title[The $\alpha$-orthogonal complements]{The $\alpha$-orthogonal complements of regular subspaces of 1-dim Brownian motion}

\author{ Liping Li}
\address{School of Mathematical Sciences, Fudan University, Shanghai 200433, China.}
\email{lipingli10@fudan.edu.cn}

\author{Xiucui Song}
\address{School of Mathematical Sciences, Fudan University, Shanghai 200433, China.}
\email{xiucuisong12@fudan.edu.cn}





\begin{abstract}
  Roughly speaking, a regular subspace of a Dirichlet form is a subspace, which is also a regular Dirichlet form, on the same state space. In particular, the domain of regular subspace is a closed subspace of the Hilbert space induced by the domain  and $\alpha$-inner product of original Dirichlet form. We shall investigate the orthogonal complement of regular subspace of 1-dimensional Brownian motion in this paper. Our main results indicate that  this orthogonal complement has a very close connection with the $\alpha$-harmonic equation under Neumann boundary condition.
\end{abstract}

\maketitle

\section{Introduction}

The regular subspace of a Dirichlet form is also a regular Dirichlet form on the same Hilbert space. It inherits the same form of original Dirichlet form but possesses a smaller domain. This conception was first raised by Fang, Fukushima and Ying in \cite{FMG} and they characterized all regular subspaces of 1-dim Brownian motion. Then the second author and his co-authors did a series of works on regular subspaces of general Dirichlet forms in \cite{XPJ} \cite{LY2} and \cite{LY}. In \cite{LY2} the authors introduced the trace Dirichlet forms to analyse the structure of regular subspaces of 1-dim Brownian motion. In particular, they decomposed the associated extended Dirichlet space $H^1_\mathrm{e}(\mathbf{R})$ of 1-dim Brownian motion into two parts: the first one is the extended Dirichlet space of  regular subspace and the second one is the pseudo orthogonal complement of regular subspace relative to the energy form of 1-dim Brownian motion. In this paper we shall extend this decomposition to the cases relative to $\alpha$-norm for any $\alpha\geq 0$ and concern their analytical representations.

We refer the terminologies of Dirichlet forms to \cite{CM} and \cite{FU}. Let $E$ be a locally compact separable metric space and $m$ a Radon measure on $E$ fully supported on $E$. Further let $(\mathcal{E}^1,\mathcal{F}^1)$ and $(\mathcal{E}^2,\mathcal{F}^2)$ be two regular Dirichlet forms on $L^2(E,m)$. Then $(\mathcal{E}^1,\mathcal{F}^1)$ is called a \emph{regular subspace} of $(\mathcal{E}^2,\mathcal{F}^2)$ if 
\begin{equation}
	\mathcal{F}^1\subset \mathcal{F}^2,\quad \mathcal{E}^2(u,v)=\mathcal{E}^1(u,v),\quad u,v\in \mathcal{F}^1.
\end{equation}
In particular, if $\mathcal{F}^1$ is a proper subset of $\mathcal{F}^2$, then $(\mathcal{E}^1,\mathcal{F}^1)$ is called a proper regular subspace of $(\mathcal{E}^2,\mathcal{F}^2)$.

It is well known that the associated Dirichlet form of $1$-dim Brownian motion on $L^2(\mathbf{R})$ is 
\[
	(\mathcal{E,F})=(\frac{1}{2}\mathbf{D},H^1(\mathbf{R})),
\]
where $H^1(\mathbf{R})$ is $1$-Sobolev space and for any $u,v\in H^1(\mathbf{R})$, $\mathbf{D}(u,v)=\int_\mathbf{R} u'(x)v'(x)dx$. Fix a constant $\alpha> 0$, 
\[
	\mathcal{E}_\alpha(u,v):=\mathcal{E}(u,v)+\alpha(u,v),\quad u,v\in \mathcal{F},
\]
where $(\cdot,\cdot)$ is the inner product of $L^2(\mathbf{R})$. Note that $\mathcal{F}$ is a Hilbert space with respect to the inner product $\mathcal{E}_\alpha$. As outlined in \cite{FMG} and \cite{LY2}, each regular subspace, denoted by $(\mathcal{E}^{(s)},\mathcal{F}^{(s)})$, of $(\mathcal{E,F})$ can be characterized by a strictly increasing and absolutely continuous function $s$ on $\mathbf{R}$ satisfying
\begin{equation}
	s'(x)=0\text{ or }1,\quad \text{a.e.}
\end{equation}
and written as
\begin{equation}\label{SBM}
 \begin{aligned}
       & \mathcal{F}^{(s)}:= \{u\in L^2(\mathbf{R}):\;u\ll s,\;\int_{\mathbf{R}}\left(\frac{du}{ds}\right)^2ds<\infty\}, \\
      &  \mathcal{E}^{(s)}(u,v):= \frac{1}{2}\int_{\mathbf{R}} \frac{du}{ds}\frac{dv}{ds}ds,\quad u,v\in \mathcal{F}^{(s)},
        \end{aligned}
    \end{equation}
   where $u\ll s$ means $u$ is absolutely continuous with respect to $s$. The function $s$ is usually called the scaling function of $(\mathcal{E}^{(s)},\mathcal{F}^{(s)})$, see \cite{CM} and \cite{RW}. Denote the extended Dirichlet spaces of $(\mathcal{E,F})$ and $(\mathcal{E}^{(s)},\mathcal{F}^{(s)})$ by $\mathcal{F}_\mathrm{e}$ and $\mathcal{F}^{(s)}_\mathrm{e}$. Define
 \begin{equation}\label{EQGXM}
 G:=\{x\in \mathbf{R}: s'(x)=1\}.
 \end{equation}
Clearly $G$ satisfies that 
\begin{equation}\label{EQMGCA}
m(G\cap (a,b))>0
\end{equation}
 for any open interval $(a,b)$. The regular subspace $(\mathcal{E}^{(s)},\mathcal{F}^{(s)})$ is a proper one if and only if the Lebesgue measure of $F:=G^c$ is positive. One may see that each regular subspace of 1-dimensional Brownian motion is also characterized by a set $G$ satisfying \eqref{EQMGCA}.
 Throughout this paper we shall always make the same assumption as \cite{LY2}:
\begin{description}
\item[(\textbf{H})] $G$ is open. 
\end{description}
Note that $G$ is defined in the sense of almost everywhere. Thus the above assumption means $G$ has an open  version. Since the typical example of $G$ is the complement of generalized Cantor type set, this assumption is reasonable. Under the assumption (\textbf{H}), one can deduce from \eqref{EQMGCA} that $F$ is the boundary of open set $G$. We also refer more notes about this assumption to \cite{LY2}. In particular, we may write 
\begin{equation}
G=\cup_{n=1}^\infty I_n,
\end{equation}
where $\{I_n=(a_n,b_n):n\geq 1\}$ is a series of disjoint open intervals.

In \cite{LY2} the authors investigated the significance of pseudo orthogonal complement of $\mathcal{F}^{(s)}_\mathrm{e}$ in $\mathcal{F}_\mathrm{e}$ relative to $\mathcal{E}$. In this paper we shall research the `real' orthogonal complement of regular subspace relative to the inner product $\mathcal{E}_\alpha$. Our main results imply that this orthogonal complement has a very close connection with $\alpha$-harmonic equation under Neumann boundary condition.

\section{The $\alpha$-orthogonal complement of $(\mathcal{E}^{(s)},\mathcal{F}^{(s)})$}\label{S2}

In \cite{LY2}, the authors defined the pseudo orthogonal complement of $\mathcal{F}_\mathrm{e}^{(s)}$ in $\mathcal{F}_\mathrm{e}$ relative to $\mathcal{E}$ formally by
\begin{equation}\label{EQMGS}
\mathcal{G}^{(s)}:=\{u\in \mathcal{F}_\text{e}:\mathcal{E}(u,v)=0\text{ for any }v\in \mathcal{F}^{(s)}_\text{e}\}.
\end{equation}
It indicates in Theorem~3.1 of \cite{LY2} that 
\begin{equation}\label{EQGSB}
	\mathcal{G}^{(s)}=\big\{u\in \mathcal{F}_\text{e}:u'\text{ is a constant a.e. on }G\big\},
\end{equation}
and any function $u\in \mathcal{F}_\mathrm{e}$ can be written as
\begin{equation}\label{EQUUU}
	u=u_1+u_2,
\end{equation}
where $u_1\in \mathcal{F}^{(s)}_\mathrm{e}$ and $u_2\in \mathcal{G}^{(s)}$. Since the inner product $\mathcal{E}$ is not complete, the decomposition \eqref{EQUUU} may only be unique up to a constant. 

In this section let us consider the complete inner product $\mathcal{E}_\alpha$ for a fixed constant $\alpha>0$. Note that $\mathcal{F}$ and $\mathcal{F}^{(s)}$ are both Hilbert spaces relative to the inner product $\mathcal{E}_\alpha$. In other words, $\mathcal{F}^{(s)}$ is a closed subspace of $\mathcal{F}$ relative to $\mathcal{E}_\alpha$. Hence we can define the natural orthogonal complement of $\mathcal{F}^{(s)}$ in $\mathcal{F}$ relative to $\mathcal{E}_\alpha$ by
\begin{equation}
	\mathcal{G}^{(s)}_\alpha:=\{u\in \mathcal{F}:\mathcal{E}_\alpha(u,v)=0\text{ for any }v\in \mathcal{F}^{(s)}\}. 
\end{equation}
When $\alpha=0$, the similar definition is represented by \eqref{EQMGS}. We may also write
\begin{equation}\label{EQFFS}
\mathcal{F}=\mathcal{F}^{(s)}\oplus_{\mathcal{E}_\alpha} \mathcal{G}^{(s)}_\alpha.
\end{equation} 
Clearly any function $u\in \mathcal{F}$ can be uniquely expressed as a sum of two functions in $\mathcal{F}^{(s)}$ and $\mathcal{G}^{(s)}_\alpha$ respectively. Denote the $\mathcal{G}^{(s)}_\alpha$-part of $u$ in this decomposition by $P_{\mathcal{G}^{(s)}_\alpha} u$. Then this decomposition can be written as
\begin{equation}\label{EQUUP}
	u=(u-P_{\mathcal{G}^{(s)}_\alpha} u)+P_{\mathcal{G}^{(s)}_\alpha} u.
\end{equation}
In particular, $u-P_{\mathcal{G}^{(s)}_\alpha} u\in \mathcal{F}^{(s)}$. 

Now we take a position to discuss the connection between the decomposition \eqref{EQFFS} and another decomposition induced by the part Dirichlet form and reduced function class. Recall that $G$ is an open set defined by \eqref{EQGXM} and the part Dirichlet form of $(\mathcal{E}^{(s)},\mathcal{F}^{(s)})$ on $G$ is defined by 
\[
	\begin{aligned}
		&\mathcal{F}^{(s)}_G:=\{u\in \mathcal{F}^{(s)}: u(x)=0,\; x\in F\},\\
		&\mathcal{E}^{(s)}_G(u,v):=\mathcal{E}^{(s)}(u,v),\quad u,v\in \mathcal{F}^{(s)}_G.
	\end{aligned}
\]
Similarly we can write the part Dirichlet form $(\mathcal{E}_G,\mathcal{F}_G)$ of $(\mathcal{E,F})$ on $G$. They are both regular Dirichlet forms on $L^2(G)$. Moreover, $\mathcal{F}_G$ (resp. $\mathcal{F}^{(s)}_G$) is a closed subspace of $\mathcal{F}$ (resp. $\mathcal{F}^{(s)}$). Their orthogonal complements are denoted by $\mathcal{H}^{\alpha}_F$ and $\mathcal{H}^{(s),\alpha}_F$, i.e. 
\begin{equation}\label{EQFFG}
\mathcal{F}=\mathcal{F}_G\oplus_{\mathcal{E}_\alpha} \mathcal{H}^\alpha_F,
\end{equation}
\begin{equation}\label{EQFFE}
\mathcal{F}^{(s)}=\mathcal{F}^{(s)}_G\oplus_{\mathcal{E}_\alpha}\mathcal{H}^{(s),\alpha}_F.
\end{equation}
For any $u\in \mathcal{F}$ (resp. $u\in \mathcal{F}^{(s)}$), its $\mathcal{H}^\alpha_F$-part (resp. $\mathcal{H}^{(s),\alpha}_F$-part) in the decomposition \eqref{EQFFG} (resp. \eqref{EQFFE}) may also be written as $H^\alpha_Fu$ (resp. $H^{(s),\alpha}_Fu$). 
The following lemma is taken directly from Lemma~2.1 and 2.2 of \cite{LY2}. Recall that $F$ is the complement of $G$. 

\begin{lemma}\label{Lemma1}\quad
It holds that
 \[
\mathcal{F}^{(s)}=\{u\in \mathcal{F}: u'=0\text{ a.e. on }F\}.\]
 Furthermore $(\mathcal{E}_G,\mathcal{F}_G)=(\mathcal{E}^{(s)}_G,\mathcal{F}^{(s)}_G)$. 
\end{lemma}

The significant result of above lemma is $\mathcal{F}_G=\mathcal{F}^{(s)}_G$. As a sequel, the two decompositions \eqref{EQFFG} and \eqref{EQFFE} have a common component. On the other hand, the orthogonal complement of $\mathcal{F}^{(s)}$ in $\mathcal{F}$ is exactly $\mathcal{G}^{(s)}_\alpha$. Thus we can deduce that
\[
	\mathcal{F}=\mathcal{F}^{(s)}_G\oplus_{\mathcal{E}_\alpha} \mathcal{H}^{(s),\alpha}_F \oplus_{\mathcal{E}_\alpha} \mathcal{G}^{(s)}_\alpha=\mathcal{F}_G\oplus_{\mathcal{E}_\alpha} \mathcal{H}^{(s),\alpha}_F \oplus_{\mathcal{E}_\alpha} \mathcal{G}^{(s)}_\alpha
\]
and
\begin{equation}\label{EQHFH}
\mathcal{H}^\alpha_F=\mathcal{H}^{(s),\alpha}_F \oplus_{\mathcal{E}_\alpha} \mathcal{G}^{(s)}_\alpha.
\end{equation}
Note that the special case $\alpha=0$ has already been discussed in \cite{LY2}. 

The following theorem is our main result of this section which obtains an expression of $\mathcal{G}^{(s)}_\alpha$. Note that it is an extension of Theorem~3.1 of \cite{LY2}. 

\begin{theorem}\label{THM1}\quad
Fix a constant $\alpha\geq 0$. Then 
	\begin{equation}\label{EQGSA}
		\mathcal{G}^{(s)}_\alpha=\{u\in \mathcal{F}:u'(x)-u'(y)=2\alpha\int_y^x u(z)dz,\quad \text{a.e. }x,y\in G\}.
	\end{equation}
	In particular, when $\alpha=0$, the above formula has the same form as \eqref{EQGSB}.
\end{theorem}
\begin{proof}\quad
Denote the set of right side of \eqref{EQGSA} by $\mathcal{G}$. Fix two functions $u\in \mathcal{F}$ and $v=\phi\circ s$ with some function $\phi\in C_c^\infty(s(\mathbf{R}))$, where $s(\mathbf{R})=\{s(x):x\in \mathbf{R}\}$ is an open interval of $\mathbf{R}$. Suppose $\text{supp}[v]\subset I$ with some interval $I=(a,b)$. Clearly
	\[
		v(x)=-\int_x^b \phi'(s(y))ds(y).
	\]
Then
\[
\begin{aligned}
	\int_\mathbf{R} u(x)v(x)dx&=-\int_a^b u(x)dx	\int_x^b \phi'(s(y))ds(y)\\
		&=-\int_a^b \phi'(s(y))ds(y)\int_a^y u(x)dx,
\end{aligned}\]
and thus
\begin{equation}\label{EQMEU}
\begin{aligned}
	\mathcal{E}_\alpha(u,v)&=-\alpha\int_a^b \phi'(s(y))ds(y)\int_a^y u(x)dx+\frac{1}{2}\int_a^b u'(y)\phi'(s(y))ds(y)\\
	&=\int_a^b \phi'(s(y))ds(y)\bigg [\frac{1}{2}u'(y)-\alpha\int_a^yu(x)dx\bigg ].
\end{aligned}\end{equation}
Now assume that $u\in \mathcal{G}$. Then for a.e. $x,y\in I\cap G$, we have 
\[
	\frac{1}{2} u'(x)-\alpha\int_a^x u(z)dx=\frac{1}{2}u'(y)-\alpha\int_a^y u(z)dz\equiv C,
\]
where $C$ is a constant. It follows from \eqref{EQGXM} and \eqref{EQMEU} that
\[
	\mathcal{E}_\alpha(u,v)=C\cdot \int_a^b \phi'(s(x))ds(x)=C\cdot \int_{s(I)} \phi'(x)dx=0.
\]
Since $C_c^\infty\circ s:=\{\phi\circ s:\phi\in C_c^\infty(s(\mathbf{R}))\}$ is $\mathcal{E}^{(s)}_1$-dense in $\mathcal{F}^{(s)}$ (see \cite{CM} and \cite{XPJ}), we can deduce that $u\in \mathcal{G}^{(s)}_\alpha$. Therefore 
\[
	\mathcal{G}\subset \mathcal{G}^{(s)}_\alpha.
\] 
On the contrary, assume $u\in \mathcal{G}^{(s)}_\alpha$. Then 
\[
	\mathcal{E}_\alpha(u,v)=0
\]
for any $v=\phi\circ s$ with $\phi\in C_c^\infty(s(\mathbf{R}))$. We still assume $\text{supp}[v]\subset I$. Let $t$ be the inverse function of $s$, i.e. $t=s^{-1}$ and 
\[
	h(y):=\frac{1}{2}u'(y)-\alpha\int_a^yu(x)dx,\quad y\in I.
\]
Then 
\[
	\mathcal{E}_\alpha(u,v)=\int_{s(I)} \phi'(x)h(t(x))dx=0
\]
for any $\phi\in C_c^\infty(s(I))$. Thus $h(t(x))$ is a constant for a.e. $x\in s(I)$ (see \cite{AF}). Denote all $x\in s(I)$ such that $h(t(x))$ is a constant by $H$ and $\tilde{H}:=t(H)$. Then $h$ is a constant on $\tilde{H}$. We claim that the Lebesgue measure of $(I\setminus \tilde{H})\cap G$ is zero, in other words, $h$ is a constant a.e. on $I\cap G$. In fact, its Lebesgue measure
\[
	|(I\setminus \tilde{H})\cap G|=\int_{I\setminus \tilde{H}} 1_G(x)dx=\int_{I\setminus \tilde{H}}ds=|s(I)\setminus H|=0.
\]
It follows from $h(x)=h(y)$ for a.e. $x,y$ on $I\cap G$ that 
\[
	u'(x)-u'(y)=2\alpha\int_y^xu(z)dz
\]
for a.e. $x,y \in I\cap G$. Since $I$ can be taken as arbitrary open interval of $\mathbf{R}$, we can deduce that $u\in \mathcal{G}$. Therefore $\mathcal{G}^{(s)}_\alpha\subset \mathcal{G}$. That completes the proof.
\end{proof}

\section{The characteristic equation}

Fix a constant $\alpha>0$ and take a function $u$ in $\mathcal{G}^{(s)}_\alpha$. Recall that $G=\cup_{n=1}^\infty I_n$ where $\{I_n:n\geq 1\}$ is a series of disjoint open intervals. For any $n\geq 1$, it follows from Theorem~\ref{THM1} that
\[
	u'(x)-u'(y)=2\alpha\int_y^x u(z)dz, \quad x,y \in I_n.
\]
Hence a version of $u'$, which is still denoted by $u'$, is absolutely continuous on $I_n$ and 
\[
	\frac{1}{2}u''(x)=\alpha u(x),\quad \text{a.e. }x\in I_n. 
\]
Since $n$ is arbitrary, we conclude that a version of $u'$, which is still denoted by $u'$ , is absolutely continuous on $G$ and 
\begin{equation}\label{EQUXA}
	\frac{1}{2}u''(x)=\alpha u(x),\quad \text{a.e. }x\in G. 
\end{equation}
It is well known that the harmonic equation $1/2u''=\alpha u$ has an essential connection with the Brownian motion, see \cite{IM}. Moreover for high-dimensional Brownian motions, the solutions to $\alpha$-harmonic equation
\[
	\frac{1}{2}\Delta u=\alpha u
\]
on a domain $D$ completely characterize the trace of Brownian motion on the boundary $\partial D$ of $D$, see Example~1.2.3 of \cite{FU}. On the other hand, as outlined in Lemma~\ref{Lemma1} the part Dirichlet form of regular subspace $(\mathcal{E}^{(s)},\mathcal{F}^{(s)})$ on $G$ is the same as that of $(\mathcal{E,F})$. That means the difference, roughly speaking some feature of $\mathcal{G}^{(s)}_\alpha$, between the regular subspace and 1-dim Brownian motion concentrates on $F=G^c$, i.e. the boundary of $G$. In the rest of this section we shall explain the connection between $\mathcal{G}^{(s)}_\alpha$ and characteristic equation \eqref{EQUXA}, but unfortunately, \eqref{EQUXA} is only a necessary (not sufficient) condition for a function being in $\mathcal{G}^{(s)}_\alpha$. Loosely speaking, the insufficiency comes from a component $\mathcal{H}^{(s),\alpha}_F$ in \eqref{EQHFH}, i.e. the difference between $\mathcal{H}^\alpha_F$ and $\mathcal{G}^{(s)}_\alpha$. 

For a given function $f\in \mathcal{F}(=H^1(\mathbf{R}))$, another function $u$ is said to be a solution to the equation \eqref{EQUXA} with the Neumann boundary condition
\begin{equation}\label{EQUFA}
u'(x)=f'(x),\quad \text{a.e. } x\in F,
\end{equation}
if $u\in H^1(\mathbf{R})$, a version of $u'$ is absolutely continuous on $G$ and $u$ satisfies \eqref{EQUXA} and \eqref{EQUFA}. Clearly the solutions to equation \eqref{EQUXA} with the Neumann boundary condition \eqref{EQUFA} always exist. In fact, it follows from the discussions above and Theorem~\ref{THM1} that $$u:=P_{\mathcal{G}^{(s)}_\alpha}f$$ satisfies all conditions except for \eqref{EQUFA}. From \eqref{EQUUP} and Lemma~\ref{Lemma1} we conclude that $f-u\in \mathcal{F}^{(s)}$ and hence $(f-u)'=0$ a.e. on $F$. Therefore $u$ also satisfies \eqref{EQUFA}, and it is actually a solution to equation \eqref{EQUXA} with the Neumann boundary condition \eqref{EQUFA}. The following theorem illustrates that the solutions are not unique and we also give all solution to \eqref{EQUXA} and \eqref{EQUFA}. 

\begin{theorem}\label{THM2}\quad 
Fix a function $f\in H^1(\mathbf{R})$. All solutions to the equation \eqref{EQUXA} with the Neumann boundary condition \eqref{EQUFA} are
\[
	\{P_{\mathcal{G}^{(s)}_\alpha}f+h:h\in \mathcal{H}^{(s),\alpha}_F\}.
\]
\end{theorem}
\begin{proof}
It is equivalent to prove that all solutions to equation \eqref{EQUXA} with the Neumann boundary condition
\begin{equation}\label{EQUQT}
	u'=0,\quad \text{a.e. }x\in F
\end{equation}
are exactly $\mathcal{H}^{(s),\alpha}_F$. First assume $u$ is a solution to equation \eqref{EQUXA} and \eqref{EQUQT}. It follows from Lemma~\ref{Lemma1} that $u\in \mathcal{F}^{(s)}$. Thus it suffices to prove that $u$ is $\mathcal{E}_\alpha$-orthogonal to every function in $C_c^\infty(G)$. To this end, for any $\phi\in C_c^\infty(G)$ it follows from \eqref{EQUXA} that
\[
\begin{aligned}
	\mathcal{E}_\alpha(u,\phi)&=\alpha\int u(x)\phi(x)dx+\frac{1}{2}\int u'(x)\phi'(x)dx	\\
		&=\int (\alpha u(x)-\frac{1}{2}u''(x))\phi(x)dx \\
		&=0. 
\end{aligned}\]
Therefore $u\in \mathcal{H}^{(s),\alpha}_F$. On the contrary, assume $u\in \mathcal{H}^{(s),\alpha}_F$. Since $\mathcal{H}^{(s),\alpha}_F\subset \mathcal{F}^{(s)}$, it follows from Lemma~\ref{Lemma1} that $u$ satisfies \eqref{EQUQT}. Note that $G$ is composed by a series of open intervals. Let $I=(a,b)$ be one of these intervals. For any $\phi\in C_c^\infty(I)\subset H^1_0(G)$,  from 
\[
	\phi(x)=\int_a^x \phi'(t)dt
\]
we obtain
\[
	0=\mathcal{E}_\alpha(u,\phi)=\int \phi'(t)\bigg( \frac{1}{2}u'(t)+\alpha\int_t^b u(x)dx\bigg)dt.
\]
Thus $1/2\cdot u'(t)+\alpha\int_t^b u(x)dx$ is a constant a.e. on $I$. We can deduce that a version of $u'$ (still denoted by $u'$) satisfies
\[
	\frac{1}{2}(u'(t_1)-u'(t_2))=\alpha\int_{t_2}^{t_1} u(x)dx,\quad t_1,t_2\in I.
\]
Hence $u'$ is absolutely continuous on $I$ and
\[
	\frac{1}{2}u''(x)=\alpha u(x),\quad \text{a.e. }x\in I.
\]
Then it follows that $u$ is a solution to equation \eqref{EQUXA} and \eqref{EQUQT}. That completes the proof.
\end{proof}

Finally we shall give an analogical result of Theorem~\ref{THM2} for $\alpha=0$. Fix a function $f\in H^1_\text{e}(\mathbf{R})$, where $$H^1_\mathrm{e}(\mathbf{R}):=\{f: f\text{ is absolutelty continuous and }f'\in L^2(\mathbf{R})\},$$ and consider the equation
\begin{equation}\label{EQUQTA}
	u''=0\quad \text{a.e. on }G
\end{equation}
with the Neumann boundary condition
\begin{equation}\label{EQUFQ}
	u'(x)=f'(x),\quad \text{a.e. }x\in F.
\end{equation}
A function $u\in H^1_\text{e}(\mathbf{R})$ is said to be a solution to equation \eqref{EQUQTA} with the Neumann boundary condition \eqref{EQUFQ} if a version of $u'$ is absolutely continuous on $G$ and $u$ satisfies \eqref{EQUQTA} and \eqref{EQUFQ}. Assume $P_{\mathcal{G}^{(s)}}f$ is (one of) $\mathcal{G}^{(s)}$-part in the orthogonal decomposition \eqref{EQUUU} with respect to $f$. Note that $P_{\mathcal{G}^{(s)}}f$ can be taken uniquely when $(\mathcal{E}^{(s)},\mathcal{F}^{(s)})$ is transient. Otherwise, it is unique up to a constant (see Theorem~3.1 of \cite{LY2}). It follows from \eqref{EQGSB}, \eqref{EQUUU} and Lemma~\ref{Lemma1} that $P_{\mathcal{G}^{(s)}}f$ is a special solution to equation \eqref{EQUQTA} with the Neumann boundary condition \eqref{EQUFQ}. 

Before giving all solutions to equation \eqref{EQUQTA} with the Neumann boundary condition \eqref{EQUFQ}, we need to make some notes. 
Clearly every constant function belongs to $H^1_\mathrm{e}(\mathbf{R})$ and satisfies \eqref{EQUQTA} and the boundary condition 
\begin{equation}\label{EQUAXF}
u'=0,\quad \text{a.e. } x\in F.
\end{equation}
On the other hand, one can easily prove that the scaling function  $s\in H^1_\text{e}(\mathbf{R})$ if and only if $s(-\infty)>-\infty$ and $s(\infty)<\infty$, equivalently the Lebesgue measure of $G$ is finite. Since $s'=1$ on $G$ and $s'=0$ on $F$, we may conclude $s$ also satisfies \eqref{EQUQTA} and \eqref{EQUAXF}. Note that $\mathcal{H}^{(s)}_F$ (resp. $\mathcal{H}_F$) is the second component in the analogical decomposition \eqref{EQFFE} (resp. \eqref{EQFFG}) for $\alpha=0$ and we refer more details to \cite{FU} and \S3 of \cite{LY2}. 

\begin{theorem}\quad 
Fix a function $f\in H^1_\mathrm{e}(\mathbf{R})$.	When $s(-\infty)>-\infty$ and $s(\infty)<\infty$, all solutions to equation \eqref{EQUQTA} with the Neumann boundary condition \eqref{EQUFQ} are
	\[
		\{P_{\mathcal{G}^{(s)}}f+h+C_1s+C_0:h\in \mathcal{H}^{(s)}_F,C_1,C_0\text{ are two constants}\}.
	\]
	Otherwise all solutions to equation \eqref{EQUQTA} with the Neumann boundary condition \eqref{EQUFQ} are
	\[
		\{P_{\mathcal{G}^{(s)}}f+h+C_0:h\in \mathcal{H}^{(s)}_F,C_0\text{ is a constant}\}.
	\]
\end{theorem}
\begin{proof}
	Similar to Theorem~\ref{THM2}, all functions in two classes above are  the solutions to  equation \eqref{EQUQTA} with the Neumann boundary condition \eqref{EQUFQ}. We only need to prove any solution $u$ can be expressed as the above form. Let 
	\[
		g:=u-P_{\mathcal{G}^{(s)}}f.
	\]
Then a version of $g'$ is still absolutely continuous and 
\begin{equation}\label{EQQQT}
	g''(x)=0\quad \text{a.e. } x\in G,	
\end{equation}
\begin{equation}\label{EQGQT}
g'(x)=0\quad \text{a.e. }x\in F.
\end{equation}
Similar to Theorem~\ref{THM2},  we may obtain $g\in \mathcal{H}_{F}$. Then it follows from Proposition~3.1 of \cite{LY2} that $g$ can be expressed as
\[
	g=g_1+g_2
\]
for some $g_1\in \mathcal{H}^{(s)}_F$ and $g_2\in \mathcal{G}^{(s)}$. Apparently $g_1$ satisfies \eqref{EQQQT} and \eqref{EQGQT}. Hence so does $g_2$. From \eqref{EQGSB} and $g_2\in \mathcal{G}^{(s)}$, we can deduce that there exists a constant $C_1$ such that 
\[
	g'_2=C_1\cdot 1_G=C_1s'\quad \text{a.e.}
\]
As a sequel $g_2=C_1s+C_0$ for another constant $C_0$. Recall that any constant function belongs to $ H^1_\text{e}(\mathbf{R})$. Moreover $s\in H^1_\text{e}(\mathbf{R})$ if and only if  $s(-\infty)>-\infty$ and $s(\infty)<\infty$. That completes the proof.
\end{proof}

\section*{Acknowledgement}
The authors would like to thank Professor Jiangang Ying for many helpful discussions.

\end{document}